\documentclass[12pt,twoside, final]{amsart}

\usepackage{amsmath,amsthm,amscd,amsfonts,amssymb,enumerate}
\usepackage[all]{xy}
\usepackage{graphicx}
\usepackage{color}
\usepackage[colorlinks]{hyperref}
\usepackage{amsfonts,amssymb,amscd,amsmath,enumerate,url,verbatim}
 \usepackage[dvips]{epsfig}
 \usepackage[none]{hyphenat}
\usepackage{amsmath,amssymb,amsfonts,enumerate,amsthm}
 \usepackage{amsgen, amstext,amsbsy,amsopn, amsthm, amsfonts,amssymb,amscd,amsmath,euscript,enumerate,url,verbatim,calc,xypic}
 \usepackage{latexsym}
 \usepackage{graphics}
 \usepackage{color}
\numberwithin{equation}{section}
\newtheorem{theorem}{Theorem}[section]
\newtheorem{lemma}[theorem]{Lemma}
\newtheorem{corollary}[theorem]{Corollary}
\theoremstyle{definition}
\newtheorem{definition}[theorem]{Definition}

\newcommand{\Ass}{\operatorname{Ass}}
\newcommand{\Ext}{\operatorname{Ext}}

\newcommand{\Hom}{\operatorname{Hom}}

\newcommand{\NN}{\mathbb N_0}

\title{Koszul Cohomology of \v{C}ech Cohomology Modules}
\author{Maryam Jahangiri}
\address{ Department of Mathematics, Faculty of Mathematical Sciences and Computer, Kharazmi University, Tehran,
Iran.}
\email{ jahangiri@khu.ac.ir\\
jahangiri.maryam@gmail.com}
\begin{document}
\maketitle

\begin{abstract}
Let $R$ be a commutative Noetherian ring, let
$\mathbf{x}=x_1,\ldots,x_n$ be an $R$-regular sequence, and let
$\mathbf{y}=y_1,\ldots,y_m$ be a sequence of elements of $R$. Put
$I=(\mathbf y)$. Let $\mathcal S$ be a Serre subcategory of the
category of $R$-modules. We consider the double complex obtained from
the Koszul co-complex with respect to $\mathbf x$ and the \v{C}ech
complex with respect to $\mathbf y$. Using the two spectral sequences
associated with this double complex, we prove that
\[
\Ext_R^i(R/(\mathbf x),H_I^j(R))\in\mathcal S
\quad\text{for all }i,j\in\NN
\]
implies
\[
H_I^j(R/(\mathbf x))\in\mathcal S
\quad\text{for all }j\in\NN.
\]
\end{abstract}

\section{Introduction}

Throughout this paper, $R$ is a commutative Noetherian ring with
$1\ne0$, and $I$ is an ideal of $R$. We denote by $\NN$ the set of
non-negative integers.

For an $R$-module $M$, the $I$-torsion submodule of $M$ is
\[
\Gamma_I(M)
=
\bigcup_{t\geq 0}(0:_M I^t)
=
\{x\in M\mid I^t x=0\text{ for some }t\geq0\}.
\]
The $i$-th local cohomology module of $M$ with respect to $I$ is
defined by
\[
H_I^i(M)
=
H^i\bigl(\Gamma_I(E_R^\bullet(M))\bigr),
\]
where $E_R^\bullet(M)$ is an injective resolution of $M$.

Local cohomology modules may also be computed by means of Koszul
complexes, \v{C}ech complexes, and direct limits of suitable
$\Ext$-modules; see \cite[Chapters 1 and 5]{BS}.

Huneke posed several fundamental questions concerning local
cohomology modules, including the following:
\begin{enumerate}
\item When does $H_I^i(M)$ vanish?
\item When is $H_I^i(M)$ Noetherian or Artinian?
\item When is $\Ass_R(H_I^i(M))$ finite?
\end{enumerate}
See \cite{Huneke} for these and related questions.

It is known that local cohomology modules are, in general, neither
Noetherian nor Artinian, and that their sets of associated primes need
not be finite; see, for example, \cite{Singh}. These questions have
generated extensive research on vanishing and finiteness properties
of local cohomology modules; see, for example,
\cite{DoseaMiranda, EghbaliBoix, g, Lyubeznik93, Lyubeznik06}.

Let $\mathbf{x}=x_1,\ldots,x_n$ be an $R$-regular sequence and let
$\mathbf{y}=y_1,\ldots,y_m$ be a sequence of elements of $R$, with
$I=(\mathbf y)$. In this paper we consider the double complex obtained
from the Koszul co-complex with respect to $\mathbf{x}$ and the
\v{C}ech complex with respect to $\mathbf y$. Using the two spectral
sequences associated with this double complex, we prove the following
result.

\begin{theorem}\label{thm:main}
Let $\mathbf{x}=x_1,\ldots,x_n$ be an $R$-regular sequence and let
$I$ be an ideal of $R$. Let $\mathcal S$ be a Serre subcategory of
$R$-Mod. If
\[
\Ext_R^i\bigl(R/(\mathbf x),H_I^j(R)\bigr)\in\mathcal S\,\,\text{for all}\,\,i,j\in\NN
\]
  then
\[
H_I^j\bigl(R/(\mathbf x)\bigr)\in\mathcal S,\,\,\text{for all}\,j\in\NN.
\]
\end{theorem}
As consequences, we obtain corresponding results for minimax,
 weakly Laskerian, finitely generated, Artinian modules and 
 modules of dimension $\leq t$, for some $t\in \NN$.
\section{THE RESULTS}

We recall the basic constructions used in the paper.

\begin{definition}
Let $N$ be an $R$-module, let
$\mathbf x=x_1,\ldots,x_n$ and
$\mathbf y=y_1,\ldots,y_m$ be sequences of elements of $R$.
\begin{enumerate}
    \item 

The Koszul co-complex of $N$ with respect to $\mathbf x$ is
\[
K^\bullet(\mathbf x,N)
=
\Hom_R(K_\bullet(\mathbf x,R),N).
\]
Here $K_\bullet(\mathbf x,R)$ denotes the Koszul complex on
$\mathbf x$. Explicitly,
\[
0\longrightarrow R
\longrightarrow R^n
\longrightarrow\cdots\longrightarrow
\bigoplus_{1\leq i_1<\cdots<i_{k+ 1}\leq n} R
\xrightarrow{d_{k+ 1}}
\bigoplus_{1\leq i_1<\cdots<i_k\leq n} R
\cdots\rightarrow R^n\longrightarrow R
\longrightarrow0,
\]
where
\[
d_{k+1}(e_{i_1}\wedge\cdots\wedge e_{i_{k+1}})
=
\sum_{s=1}^{k+1}(-1)^{s-1}x_{i_s}
e_{i_1}\wedge\cdots\wedge
\widehat{e_{i_s}}\wedge\cdots\wedge e_{i_{k+1}}.
\]
If $\mathbf x$ is an $R$-regular sequence, then
$K_\bullet(\mathbf x,R)$ is a finite free resolution of
$R/(\mathbf x)$; see \cite[Chapter 1]{BH}.
Consequently,
\begin{equation}\label{koszul}
    H^i(K^\bullet(\mathbf x,N))
\cong
\Ext_R^i(R/(\mathbf x),N).
\end{equation}
 
\item
The \v{C}ech complex of $N$ with respect to
$\mathbf y=y_1,\ldots,y_m$ is denoted by
$\check C^\bullet(\mathbf y,N)$ and  is defined to be 
   \[0\rightarrow N \rightarrow \oplus_{i= 1}^m N_{y_i}\rightarrow\cdots \rightarrow 
   \bigoplus_{1\leq i_1<\cdots <i_t\leq m}N_{y_{i_1}\cdots y_{i_t}}\stackrel{d^t}{\longrightarrow}\bigoplus_{1\leq i_1<\cdots <i_{t+1}\leq m}N_{y_{i_1}\cdots y_{i_{t+1}}} \cdots  \rightarrow  N_{y_1 \cdots  y_m}\longrightarrow0,\]
   where for any $1\leq t\leq m$, $1\leq i_1<\cdots <i_t\leq m$ and  $1\leq j_1<\cdots <j_{t+ 1}\leq m$ 
   \[d^t:N_{y_{i_1}\cdots y_{i_t}}\longrightarrow N_{y_{j_1}\cdots y_{j_{t+ 1}}} \] is defined as
  \[
   d^i(\tfrac{z}{y_{i_1}^n\cdots y_{i_t}^n})= \begin{cases}
    (-1)^{s- 1} \tfrac{z}{y_{i_1}^n\cdots y_{i_{t }}^n}, & \text{if } (i_1,\cdots, i_t)= (j_1,\cdots, \hat{j_s},\cdots, j_{t+ 1}),\\
     0, & \text{otherwise, } 
 \end{cases}\]
 for all  $\tfrac{z}{y_{i_1}^n\cdots y_{i_t}^n}\in N_{y_{i_1}\cdots y_{i_t}}$.
The cohomologies of $\check C^\bullet(\mathbf y,N)$ computes local
cohomology:
\begin{equation}\label{ceck}
  H^i(\check C^\bullet(\mathbf y,N))
\cong H_I^i(N),
\qquad I=(\mathbf y).  
\end{equation}
 
See \cite[Theorem 5. 1. 20]{BS}.
 
\item
Let $X^{\bullet,\bullet}$ be the double complex
\[
X^{i,j}:=
\check C^j(\mathbf y,K^i(\mathbf x,R)).
\]
Equivalently, its $i$-th row is the \v{C}ech complex of the $i$-th
module in $K^\bullet(\mathbf x,R)$, while its $j$-th column is the
Koszul co-complex of the $j$-th module in
$\check C^\bullet(\mathbf y,R)$.

\[
\begin{array}{ccccccccc}

 & & 0 & & 0 & &   & & 0 \\[0.5em]
 & & \uparrow & & \uparrow & & & & \uparrow \\

0 & \longrightarrow &
R
& \longrightarrow &
\displaystyle\oplus_{i=1}^{m}R_{y_i}
& \longrightarrow &
\cdots
& \longrightarrow &
R_{y_1\cdots y_m}
\longrightarrow 0
\\[1.2em]

 & & \uparrow & & \uparrow & & & & \uparrow \\

0 & \longrightarrow &
\displaystyle\oplus_{j=1}^{n}R
& \longrightarrow &
\displaystyle\oplus_{j=1}^{n}
\left(\oplus_{i=1}^{m}R_{y_i}\right)
& \longrightarrow &
\cdots
& \longrightarrow &
\displaystyle\oplus_{j=1}^{n}R_{y_1\cdots y_m}
\longrightarrow 0
\\

 & & \vdots & & \vdots & & & & \vdots \\

 & & \uparrow & & \uparrow & & & & \uparrow \\

0 & \longrightarrow &
\displaystyle\bigoplus_{1\le j_1<\cdots<j_s\le n}R
& \longrightarrow &
\displaystyle\bigoplus_{1\le j_1<\cdots<j_s\le n}
\left(\bigoplus_{i=1}^{m}R_{y_i}\right)
& \longrightarrow &
\cdots
& \longrightarrow &
\displaystyle\bigoplus_{1\le j_1<\cdots<j_s\le n}
R_{y_1\cdots y_m}
\longrightarrow 0
\\

 & & \uparrow & & \uparrow & & & & \uparrow \\

 & & \vdots & & \vdots & & & & \vdots \\

 & & \uparrow & & \uparrow & & & & \uparrow \\

0 & \longrightarrow &
R
& \longrightarrow &
\displaystyle\oplus_{i=1}^{m}R_{y_i}
& \longrightarrow &
\cdots
& \longrightarrow &
R_{y_1\cdots y_m}
\longrightarrow 0
\\

 & & \uparrow & & \uparrow & & & & \uparrow \\

 & & 0 & & 0 & &   & & 0
\end{array}
\]

\item
Let $\operatorname{Tot}(X^{\bullet,\bullet})$ denote its total
complex. Since the double complex is bounded in both directions, the
two standard spectral sequences converge to the cohomology of the
total complex.

For further information on this topic, we refer the reader to Chapter~10 of \cite{r}.
\end{enumerate}
\end{definition}
\begin{lemma}\label{lem:ss}
With the above notation, there are two convergent spectral sequences
\[
{}^IE_2^{i,j}
=
H^i\bigl(K^\bullet(\mathbf x,H_I^j(R))\bigr)
\Longrightarrow
H^{i+j}(\operatorname{Tot}X),
\]
and
\[
{}^{II}E_2^{i,j}
=
H_I^i\bigl(H^j(K^\bullet(\mathbf x,R))\bigr)
\Longrightarrow
H^{i+j}(\operatorname{Tot}X).
\]
\end{lemma}

\begin{proof}
Apply the standard spectral sequences associated with a bounded
double complex. Taking cohomology first in the \v{C}ech direction
gives the first spectral sequence, while taking cohomology first in
the Koszul direction gives the second one. Since
\[
H^j(\check C^\bullet(\mathbf y,R))\cong H_I^j(R)
\]
 
the displayed descriptions of the two $E_2$-pages follow.
\end{proof}

\begin{definition}
Let $\mathcal C(R)$ denote the category of all $R$-modules. A
subcategory $\mathcal S$ of $\mathcal C(R)$ is called a Serre
subcategory if, for every exact sequence
\[
0\longrightarrow M\longrightarrow N\longrightarrow L
\longrightarrow0,
\]
one has
\[
N\in\mathcal S
\quad\Longleftrightarrow\quad
M,L\in\mathcal S.
\]
\end{definition}

Examples include the classes of finitely generated modules, Artinian
modules, finite-length modules, minimax modules, and weakly
Laskerian modules. See \cite{DMA, Rudlof}.

Recall that an $R$-module $M$ is minimax if it has a finitely generated
submodule $N$ such that $M/N$ is Artinian. An $R$-module $M$ is
weakly Laskerian if 
$\Ass_R(M/N)$ 
is finite for every submodule $N$ of $M$.

\begin{theorem}\label{thm:main2}
Let $\mathbf x=x_1,\ldots,x_n$ be an $R$-regular sequence, let $I$ be
an ideal of $R$, and let $\mathcal S$ be a Serre subcategory of
$\mathcal C(R)$. If
\[
\Ext_R^i(R/(\mathbf x),H_I^j(R))\in\mathcal S
\quad\text{for all }i,j\in\NN,
\]
  then
\[
H_I^j(R/(\mathbf x))\in\mathcal S  \quad\text{for all } j\in\NN.
\]
\end{theorem}

\begin{proof}
By Lemma~\ref{lem:ss}, the first spectral sequence has
\[
{}^IE_2^{i,j}
=
H^i(K^\bullet(\mathbf x,H_I^j(R))).
\]
Since $K_\bullet(\mathbf x,R)$ is a finite free resolution of
$R/(\mathbf x)$, we have
\[
{}^IE_2^{i,j}
\cong
\Ext_R^i(R/(\mathbf x),H_I^j(R)).
\]
By hypothesis, all terms on the $E_2$-page of the first spectral
sequence belong to $\mathcal S$. Every subsequent page is obtained
from the preceding page by taking subquotients; hence
\[
{}^IE_\infty^{i,j}\in\mathcal S \quad\text{for all } i, j\in\NN.
\] 

The convergence of the spectral sequence yields a finite filtration
of $H^t(\operatorname{Tot}X)$ whose successive quotients are the
modules ${}^IE_\infty^{i,t-i}$. To be more precise, there exists a filtration
\[
0=\varphi^{\,t+1}
\subseteq
\varphi^{\,t}
\subseteq
\cdots
\subseteq
\varphi^{\,1}
\subseteq
\varphi^{\,0}
=
H^{t}(\operatorname{Tot}X),
\]
for every \(t\in\mathbb{N}_{0}\), such that
\[
{}^{\mathrm{I}}E_{\infty}^{\,i,j}
\cong
\frac{\varphi^{\,i}}{\varphi^{\,i+1}},
\qquad
\text{for all } i,j\in\mathbb{N}_{0}
\text{ with } i+j=t.
\]

Since $\mathcal S$ is a Serre
subcategory, it follows that
\[
H^t(\operatorname{Tot}X)\in\mathcal S,\,\,\text{for every}\,\ t\in \NN.
\]
On the other hand, the regularity of $\mathbf x$ gives
\[
H^j(K^\bullet(\mathbf x,R))
= \operatorname{Ext}_R^j(R/(\mathbf x), R)= 
\begin{cases}
R/(\mathbf x),&j= n,\\
0,&j\neq n.
\end{cases}
\]
Thus the second spectral sequence collapses to its $j= n$ row, and
\begin{align*}
        H_I^t(R/(\mathbf x))&\cong H_I^t(\operatorname{Ext}_R^n(R/(\mathbf x), R))\\
       &=  {}^{\mathrm{II}}E_{2}^{\,t,n}\\
       &\cong {}^{\mathrm{II}}E_{\infty}^{\,t,n} \\
     & \cong H^{t+ n}(\operatorname{Tot}X).
\end{align*}
 
Therefore
\[
H_I^t(R/(\mathbf x))\in\mathcal S,\,\,\text{for every}\,\,  t\in\NN,
\]
 as required.
\end{proof}

\begin{corollary}
Let $\mathbf x=x_1,\ldots,x_n$ be an $R$-regular sequence and let
$I$ be an ideal of $R$.

\begin{enumerate}
\item If 
$\Ext_R^i(R/(\mathbf x),H_I^j(R))$ 
is minimax (respectively, finitely generated or Artinian) for all
$i,j\in\NN$, then 
$H_I^j(R/(\mathbf x))$  
has the same property for all $j\in\NN$.

\item If 
$\Ext_R^i(R/(\mathbf x),H_I^j(R))$ 
is weakly Laskerian for all $i,j\in\NN$, then 
$H_I^j(R/(\mathbf x))$ 
is weakly Laskerian for all $j\in\NN$.

\item If, for some $t\in\NN$,  
$\dim_R\Ext_R^i(R/(\mathbf x),H_I^j(R))\le t$ 
for all $i,j\in\NN$, then 
$\dim_R H_I^j(R/(\mathbf x))\le t$ 
for all $j\in\NN$.
\end{enumerate}
\end{corollary}
Note that if $H_I^j(R)\in\mathcal{S}$ for every $j\in\mathbb{N}_0$, then
$\operatorname{Ext}_R^i\left(R/(\mathbf x),H_I^j(R)\right)\in\mathcal{S}$
for all $i,j\in\mathbb{N}_0$. Consequently, the above corollary also holds
in this case.

Indeed, this follows immediately from the long exact sequence of local
cohomology modules 
\[\cdots\longrightarrow H_I^j(R)\xrightarrow{\cdot x}H_I^j(R)
\longrightarrow H_I^j(R/xR)\longrightarrow H_I^{j+1}(R)
\longrightarrow\cdots.\]


\begin{thebibliography}{99}

\bibitem{BH}
     W. Bruns, J. Herzog, {\em Cohen-Macaulay Rings}, 2nd ed., Cambridge University Press (1998).

\bibitem{BS}
M. P. Brodmann and R. Y. Sharp,
\emph{Local Cohomology: An Algebraic Introduction with Geometric Applications},
2nd ed., Cambridge University Press (2012).

\bibitem{DMA}
K. Divaani-Aazar and A. Mafi,
{\em Associated primes of local cohomology modules of weakly Laskerian modules},
 Commun. Algebra  34 (2006)  681--690.

\bibitem{DoseaMiranda}
A. Dosea and C. B. Miranda-Neto,
{\em On Huneke's conjecture about associated primes of local cohomology modules},
 J. Algebra  669 (2025)  143--158.

\bibitem{EghbaliBoix}
M. Eghbali and A. F. Boix,
{\em Vanishing of local cohomology and set-theoretically Cohen--Macaulay ideals},
 Rend. Circ. Mat. Palermo  72 (2023)  3305--3323.

\bibitem{g}
M. Gintz, W. and Zhang, {\em Koszul cohomology and support of local cohomology modules of complete intersections}, J. London Math. Soc., 113(3) (2026). https://doi.org/10.1112/jlms.70466

\bibitem{Huneke}
C. Huneke,
{\em Problems on local cohomology},
in  Free Resolutions in Commutative Algebra and Algebraic Geometry,
Research Notes in Mathematics 2, Jones and Bartlett, (1992) 93--108.

\bibitem{Lyubeznik93}
G. Lyubeznik,
{\em Finiteness properties of local cohomology modules
(an application of D-modules to commutative algebra)},
 Invent. Math.  113 (1993) 41--55.

\bibitem{Lyubeznik06}
G. Lyubeznik,
{\em On the vanishing of local cohomology in characteristic $p>0$},
Compos. Math.  142 (2006)  207--221.


\bibitem{r}
J. J. Rotman, {\em An introduction to homological algebra}, 2nd ed. Academic press (2008).


\bibitem{Rudlof}
P. Rudlof,
{\em On Minimax and Related Modules},
 Canad. J. Math.  44 (1992) 154--166.

\bibitem{Singh}
A. K. Singh,
{\em $p$-Torsion elements in local cohomology modules},
 Math. Res. Lett.  7 (2000) 165--176.

\end{thebibliography}
\end{document}